\documentclass[11 pt]{amsart}

\usepackage{amsfonts}
\usepackage{amssymb}
\usepackage{amsmath}
\usepackage{latexsym}
\usepackage{mathrsfs}
\usepackage{amsthm}
\usepackage{verbatim}
\usepackage{amscd}

\newtheorem{thrm}{Theorem}[section]
\newtheorem{lem}[thrm]{Lemma}
\newtheorem{cor}[thrm]{Corollary}
\newtheorem{prop}[thrm]{Proposition}
\newtheorem{conj}[thrm]{Conjecture}

\theoremstyle{definition}
\newtheorem{defn}[thrm]{Definition}

\newtheorem{exmple}[thrm]{Example}
\newtheorem{rmk}[thrm]{Remark}
\newtheorem{ques}[thrm]{Question}

\newtheorem{constr}[thrm]{Construction}
\newtheorem{conv}[thrm]{Convention}

\DeclareMathOperator{\Eff}{\overline{Eff}}

\DeclareMathOperator{\Chow}{Chow}

\DeclareMathOperator{\Hilb}{Hilb}

\DeclareMathOperator{\Supp}{Supp}

\DeclareMathOperator{\rchdim}{rchdim}

\DeclareMathOperator{\ratvar}{ratvar}

\DeclareMathOperator{\var}{var}

\DeclareMathOperator{\vol}{vol}

\DeclareMathOperator{\ch}{ch}

\DeclareMathOperator{\characteristic}{char}

\DeclareMathOperator{\Pic}{Pic}

\DeclareMathOperator{\alb}{alb}

\DeclareMathOperator{\Alb}{Alb}

\DeclareMathOperator{\chdim}{chdim}

\DeclareMathOperator{\rk}{rk}

\begin{document}

\title{Asymptotic behavior of the dimension of the Chow variety}
\author{Brian Lehmann}
\thanks{The author is supported by NSF Award 1004363.}
\address{Department of Mathematics, Boston College  \\
Chestnut Hill, MA \, \, 02467}

\email{lehmannb@bc.edu}

\begin{abstract} We analyze the asymptotics of the dimension of components of the Chow variety as degree increases.  By analogy with the divisor case, the main goal is to relate the asymptotic behavior with the positivity of the corresponding cycle classes.  We also compute the dimension of the Chow variety of projective space.
\end{abstract}

\maketitle

\section{Introduction}


Let $X$ be an integral projective variety over an algebraically closed field.  Fix a numerical cycle class $\alpha$ on $X$ and let $\Chow(X,\alpha)$ denote the components of $\Chow(X)$ which parametrize cycles of class $\alpha$.  The basic feature of $\Chow(X,\alpha)$ is its dimension.  However, this can be difficult to determine due to the lack of a first-order deformation theory.  Furthermore, it is unclear how this dimension interacts with other geometric features of $\alpha$.

For divisors, a well-established principle is that the geometry of $L$ is best reflected not by $\Chow(X,[L])$ but by studying the asymptotic behavior of $\Chow(X,m[L])$ as $m$ increases.  Our goal is to establish a similar principle for arbitrary cycle classes.

We first consider the behavior of the Chow variety for the most important example: projective space.  \cite{eh92} analyzes the dimension of the Chow variety of curves on $\mathbb{P}^{n}$.  Let $\ell$ denote the class of a line in $\mathbb{P}^{n}$.  Then for $d>1$,
\begin{equation*}
\dim \Chow(\mathbb{P}^{n},d \ell) = \max \left\{ 2d(n-1), \frac{d^{2} + 3d}{2} + 3(n-2) \right\}.
\end{equation*}
The first number is the dimension of the space of unions of $d$ lines on $\mathbb{P}^{n}$, and the second is the dimension of the space of degree $d$ planar curves.  Note the basic dichotomy: in low degrees the maximal dimension is achieved by unions of linear spaces, while for sufficiently high degrees the maximal dimension is achieved by ``maximally degenerate'' irreducible curves.

We first prove the analogous statement in arbitrary dimension, establishing a conjecture of \cite{eh92}:
\begin{thrm} \label{firstthm}
Let $\alpha$ denote the class of the $k$-plane on $\mathbb{P}^{n}$.  Then for $d>1$ the dimension of $\Chow(\mathbb{P}^{n},d \alpha)$ is
\begin{align*}
\max \left\{ d(k+1)(n-k), \left( \begin{array}{c} d+k+1 \\ k+1 \end{array} \right) - 1 + (k+2)(n-k-1) \right\}.
\end{align*}
\end{thrm}
The first number is the dimension of the space of unions of $d$ $k$-planes on $\mathbb{P}^{n}$ and the second is the dimension of the space of degree $d$ hypersurfaces in $(k+1)$-planes.  The same dichotomy seen for curves also holds in higher dimensions.

The basic tool in \cite{eh92} is Castelnuovo theory for curves; however, it is unclear how best to formulate an analogue for higher dimension cycles.  Instead, we reduce the general case to the calculation for curves by taking hyperplane sections.  The key step is to bound the dimension of the fibers of the restriction map $\Chow(\mathbb{P}^{n}) \dashrightarrow \Chow(H)$.  Conceptually speaking, this approach is loosely analogous to using the LES of sheaf cohomology to understand the behavior of divisors under restriction.

We now return to the setting where $X$ is an arbitrary variety and $\alpha$ is a numerical cycle class on $X$.  Our goal is to relate the asymptotic behavior of $\Chow(X,m\alpha)$ with the geometry of $\alpha$.  As mentioned before, this relationship is one of the foundational tools of modern divisor theory.  Given a Cartier divisor $L$ on a smooth variety $X$ of dimension $n$, one defines the volume of $L$ to be
\begin{equation*}
\vol(L) := \limsup_{m \to \infty} \frac{\dim H^{0}(X,\mathcal{O}_{X}(mL))}{m^{n}/n!}.
\end{equation*}
This well-studied invariant is an important measurement of the positivity of $L$.  In particular, divisors with positive volume are precisely those whose numerical class lies in the interior of the pseudo-effective cone. 
(Note that while the linear series $|mL|$ differs from the universal family over $\Chow(X,[L])$, the discrepancy in dimension is controlled by the Picard variety and thus drops out in the asymptotic calculation; see Example \ref{variationofdivisors}.)

To understand the analogous definition for cycle classes, we first need to estimate the growth rate of $\dim \Chow(X,m\alpha)$.  The right intuition is to imagine that $m\alpha$ is a divisor class on a fixed $(k+1)$-dimensional subvariety of $X$: Theorem \ref{familydim} shows that as $m$ increases the dimension of $\Chow(X,m\alpha)$ is bounded above by $Cm^{k+1}$ for some constant $C$.  The variation function identifies the best possible constant $C$.

\begin{defn}
Let $X$ be an integral projective variety and suppose $\alpha \in N_{k}(X)_{\mathbb{Z}}$ for $0 \leq k < \dim X$.  The variation of $\alpha$ is
\begin{equation*}
\var(\alpha) = \limsup_{m \to \infty} \frac{ \dim \Chow(X,m\alpha)}{m^{k+1}/(k+1)!}.
\end{equation*}
\end{defn}

It turns out that $\var$ is homogeneous and so extends naturally to $\mathbb{Q}$-numerical classes.  In fact, Theorem \ref{bigvarcont} shows that $\var$ extends to a continuous function on the interior of $\Eff_{k}(X)$.

\begin{exmple} \label{curvesinpn}
Using Theorem \ref{firstthm}, we see that for the hyperplane class $\alpha \in N_{k}(\mathbb{P}^{n})$ we have $\var(\alpha) = 1$.
\end{exmple}

\begin{exmple} \label{p3blowup}
Let $X$ be the blow-up of $\mathbb{P}^{3}$ at a closed point.  Let $E$ denote the exceptional divisor and let $\alpha$ be the class of a line in $E$.  For any positive integer $m$, every effective cycle of class $m\alpha$ is contained in $E$.  Thus, the variation coincides with the variation of the line class $\ell$ on $\mathbb{P}^{2}$, showing that that $\var(\alpha)=1$.
\end{exmple}

The previous examples are typical: components of $\Chow(X)$ with large dimension tend to parametrize cycles that are as ``degenerate'' as possible.

To state our main theorem, we will need to recall some notions concerning the positivity of numerical cycle classes.  Let $N_{k}(X)$ denote the vector space of numerical classes of $k$-cycles on $X$ with $\mathbb{R}$-coefficients.  The pseudo-effective cone $\Eff_{k}(X) \subset N_{k}(X)$ is defined to be the closure of the cone generated by all effective $k$-cycles.  Classes that lie in the interior of the cone are known as big classes.

Note that, in contrast to the volume, even if $\alpha$ has positive variation it need not be a big class (see  Example \ref{p3blowup}).  In fact, the pushforward of a big class from a subvariety $V$ of $X$ will always have positive variation.  Our main theorem shows that this is essentially the only way to construct classes of positive variation -- variation is a measure of positivity along subvarieties.

\begin{thrm} \label{mainthrm}
Let $X$ be an integral projective variety and suppose $\alpha \in N_{k}(X)_{\mathbb{Q}}$ for $0 \leq k < \dim X$.  Then $\var(\alpha) > 0$ if and only if there is a morphism $f: Y \to X$ from an integral projective variety of dimension $k+1$ that is generically finite onto its image and a big class $\beta \in N_{k}(Y)_{\mathbb{Q}}$ such that some multiple of $\alpha - f_{*}\beta$ is represented by an effective cycle.
\end{thrm}

\begin{rmk}
As suggested by Theorem \ref{mainthrm}, there are two fundamentally different approaches to generalizing the volume function for divisors.  First, one can focus on ``deformations.''  The variation function captures this feature of cycles and is closely related to a number of classical results concerning cycles with large deformation spaces.  However, variation fails to characterize global positivity in the same as the volume function for divisors. Second, one can focus on ``positivity.''  The mobility function defined in \cite{lehmann13} shares many of the desirable geometric features of the volume function; in particular, a numerical class has positive mobility precisely when it is big.  However, the mobility is more difficult to work with explicitly and does not have any direct connection with the Chow variety.
\end{rmk}

There are two main conceptual advances used in the proof of Theorem \ref{mainthrm}.  The first is an analysis of restriction maps.  A key step in the proof of Theorem \ref{mainthrm} is using an a priori bound on asymptotic growth to find a family of cycles whose restriction to a hyperplane $H$ coincides.  Again this replaces the use of the LES of sheaf cohomology for divisors.

The second is a new characterization of bigness.  \cite[Theorem 2.4]{8authors} shows that a curve class $\alpha \in N_{1}(X)_{\mathbb{Z}}$ is big if and only if any two points of $X$ can be connected via a chain of effective cycles with numerical classes proportional to $\alpha$.  For cycles of higher dimension, the analogue is to consider chains of effective cycles that intersect in codimension $1$ along ``positive'' subvarieties.  The following definition encodes a strong version of this property.

\begin{defn}
Let $X$ be an integral projective variety of dimension $n$.  Suppose that $W$ is an integral variety and $U \subset W \times X$ is a family of effective $k$-cycles.  Denote the projection maps by $p: U \to W$ and $s: U \to X$.  We say that the family is strongly big-connecting if $s: U \to X$ is dominant and there is a big effective Cartier divisor $B$ on $X$ such that every $p$-horizontal component of $s^{*}B$ is contracted by $s$ to a subvariety of $X$ of dimension at most $k-1$.
\end{defn}

The strongly big-connecting property forces our cycles to intersect in codimension $1$ along a fixed ``positive'' $(k-1)$-dimensional subvariety (controlled by the divisor $B$).  It also implies that two general points in $X$ can be connected by a length-two chain of members of the family.  A typical example is given by fixing a $(n-k+1)$-dimensional subspace $V \subset H^{0}(X,\mathcal{O}_{X}(A))$ for a very ample divisor $A$ and taking complete intersections of $(n-k)$ general elements of $V$.  This family is strongly big-connecting for any divisor $A \in |V|$:  the intersection of $A$ with a general element of our family is contained in the base locus of $V$ which has dimension $k-1$.  The study of connecting chains leads to the following characterization of bigness for cycle classes.

\begin{thrm} \label{mainthrm2}
Let $X$ be an integral projective variety and let $\alpha \in N_{k}(X)$ for $0 \leq k < \dim X$.  Then $\alpha$ is big if and only if there is some strongly big-connecting family of $k$-cycles of class $\beta$ and a positive constant $C>0$ such that $\alpha - C\beta$ is pseudo-effective.
\end{thrm}

\begin{rmk}
\cite{voisin10} conjectures that bigness of a class $\alpha$ can be characterized using the tangency behavior of cycles representing $\alpha$ and shows that this conjecture has interesting ties to the Hodge theory of complete intersections in projective space.  Theorem \ref{mainthrm2} can be viewed as a step in this direction.
\end{rmk}

\subsection{Organization}

Section \ref{prelimsection} reviews background material on cycles.  Section \ref{cyclefamilysection} describes several geometric constructions for families of cycles.  In Section \ref{chowdimpnsec} we bound the dimension of components of $\Chow(\mathbb{P}^{n})$.  Finally, section \ref{variationsection} introduces the variation function and proves its basic geometric properties.

\subsection{Acknowledgements}

My sincere thanks to A.M.~Fulger for numerous discussions and for his suggestions.  Thanks to C.~Voisin for recommending a number of improvements on an earlier draft.  Thanks also to B.~Hassett and R.~Lazarsfeld for helpful conversations and to Z.~Zhu and X.~Zhao for reading a draft of the paper.

\section{Preliminaries} \label{prelimsection}

Throughout we work over a fixed algebraically closed field $K$.  A variety will mean a quasiprojective scheme of finite type over $K$ (which may be reducible and non-reduced).

\begin{lem} \label{dimensionlemma}
Let $X$ be an irreducible variety.  Suppose that $f: X \dashrightarrow Y$ is a rational map to a variety $Y$ and $g: X \dashrightarrow Z$ is a rational map to a variety $Z$.  Let $F$ be a general fiber of $f$ over a closed point of $Y$ (so in particular $F$ is not contained in the locus where $g$ is not defined).  Then $\dim(\overline{g(X)}) \leq \dim(\overline{g(F)}) + \dim Y$.
\end{lem}

\begin{proof}
Let $U$ be an open subset of $X$ where both $f$ and $g$ are defined and let $h: U \to Y \times Z$ be the induced map.  Since $F$ is a general fiber, $h(F \cap U)$ is dense in its closure in $Y \times Z$.  By considering the first projection, we see that $\dim(\overline{h(U)}) \leq \dim(\overline{g(F \cap U)}) + \dim Y$.
\end{proof}

We will often use the following special case of \cite[Th\'eor\`eme 5.2.2]{gr71}.

\begin{thrm}[\cite{gr71}, Th\'eor\`eme 5.2.2]
Let $f: X \to S$ be a projective morphism of varieties such that some component of $X$ dominates $S$.  There is a birational morphism $\pi: S' \to S$ such that the morphism $f': X' \to S'$ is flat, where $X' \subset X \times_{S} S'$ is the closed subscheme defined by the ideal of sections whose support does not dominate $S'$.
\end{thrm}

\subsection{Cycles}
Suppose that $X$ is a projective variety.  A $k$-cycle on $X$ is a finite formal sum $\sum a_{i} V_{i}$ where the $a_{i}$ are integers and each $V_{i}$ is an integral closed subvariety of $X$ of dimension $k$.  The cycle is said to be effective if each $a_{i} \geq 0$.  The group of $k$-cycles is denoted $Z_{k}(X)$ and the group of $k$-cycles up to rational equivalence is denoted $A_{k}(X)$.  We will follow the conventions of \cite{fulton84} in the use of various intersection products on $A_{k}(X)$.

\cite[Chapter 19]{fulton84} defines a $k$-cycle on $X$ to be numerically trivial if its rational equivalence class has vanishing intersection with every weighted homogeneous degree-$k$ polynomial in Chern classes of vector bundles on $X$.  By quotienting $Z_{k}(X)$ by numerically trivial cycles, we obtain an abelian group $N_{k}(X)_{\mathbb{Z}}$ which is finitely generated by \cite[Example 19.1.4]{fulton84}.

We also define
\begin{align*}
N_{k}(X)_{\mathbb{Q}} & := N_{k}(X)_{\mathbb{Z}} \otimes \mathbb{Q} \\
N_{k}(X) & := N_{k}(X)_{\mathbb{Z}} \otimes \mathbb{R}
\end{align*}

Suppose that $f: Z \to X$ is an l.c.i.~morphism of codimension $d$.  Then \cite[Example 19.2.3]{fulton84} shows that the Gysin homomorphism $f^{*}: A_{k}(X) \to A_{k-d}(Z)$ descends to numerical equivalence classes.  We will often use this fact when $Z$ is a Cartier divisor on $X$ to obtain maps $f^{*}: N_{k}(X) \to N_{k-1}(Z)$.

\begin{conv} \label{dimconv} When we discuss $k$-cycles on an integral projective variety $X$, we will always implicitly assume that $0 \leq k < \dim X$.  This allows us to focus on the interesting range of behaviors without repeating hypotheses.
\end{conv}

For a cycle $Z$ on $X$, we let $[Z]$ denote the numerical class of $Z$, which can be naturally thought of as an element in $N_{k}(X)_{\mathbb{Z}}$, $N_{k}(X)_{\mathbb{Q}}$, or $N_{k}(X)$.  If $\alpha$ is the class of an effective cycle $Z$, we say that $\alpha$ is an effective class.  We write $\alpha \preceq \beta$ if the difference $\beta - \alpha$ is an effective class.

\begin{defn}
Let $X$ be a projective variety.  The pseudo-effective cone $\Eff_{k}(X) \subset N_{k}(X)$ is the closure of the cone generated by all classes of effective $k$-cycles.  The big cone is the interior of the pseudo-effective cone.
\end{defn}

\cite[Theorem 0.2]{fl13} shows that $\Eff_{k}(X)$ is a full-dimensional salient cone.  For any morphism of integral projective varieties $f: X \to Y$, there is a pushforward map $f_{*}: N_{k}(X) \to N_{k}(Y)$.  \cite[Corollary 3.22]{fl13} shows that when $f$ is surjective there is an induced equality $f_{*}(\Eff_{k}(X)) = \Eff_{k}(Y)$.

\begin{conv}
Let $B$ be a Cartier divisor on an equidimensional projective variety $X$ of dimension $n$.  We say that $B$ is big if for each reduced component $X_{i}$ of $X$ we have that $h^{0}(X_{i},mB|_{X_{i}}) > \lfloor Cm^{n} \rfloor$ for some positive constant $C$.  This implies that the corresponding Weil divisor has big class.  With this definition bigness of a Cartier divisor is preserved by generically-finite pullback.
\end{conv}

We will also need a basic result concerning pseudo-effective cycles.

\begin{lem}[\cite{lehmann13}, Lemma 2.5] \label{intlem}  
Let $X$ be a projective variety and let $D$ be the support of an effective big $j$-cycle $Z$ with injection $i: D \to X$.  If $\alpha \in N_{k}(D)$ is big (for $0 \leq k \leq j$) then $i_{*}\alpha \in N_{k}(X)$ is big.
\end{lem}


\section{Families of cycles} \label{cyclefamilysection}

In this section we set up a framework for discussing families of cycles.  This framework was already discussed in \cite[Section 3]{lehmann13}; however, the explanation here is more thorough and espouses a slightly different viewpoint.

Although there are several different notions of a family of cycles in the literature, the theory we will develop is somewhat insensitive to the precise choices.  It will be most convenient to use a simple geometric definition.

\begin{defn}  \label{familydef}
Let $X$ be a projective variety.  A family of $k$-cycles on $X$ consists of an integral variety $W$, a reduced closed subscheme $U \subset W \times X$, and an integer $a_{i}$ for each component $U_{i}$ of $U$, such that for each component $U_{i}$ of $U$ the first projection map $p: U_{i} \to W$ is flat dominant of relative dimension $k$.  If each $a_{i} \geq 0$ we say that we have a family of effective cycles.  We say that $\sum a_{i}U_{i}$ is the cycle underlying the family.  

In this situation $p: U \to W$ will denote the first projection map and $s: U \to X$ will denote the second projection map unless otherwise specified.  We will usually denote a family of $k$-cycles using the notation $p: U \to W$, with the rest of the data implicit.

For a closed point $w \in W$, the base change $w \times_{W} U_{i}$ is a subscheme of $X$ of pure dimension $k$ and thus defines a fundamental $k$-cycle $Z_{i}$ on $X$.  The cycle-theoretic fiber of $p: U \to W$ over $w$ is defined to be the cycle $\sum a_{i}Z_{i}$ on $X$.  We will also call these cycles the members of the family $p$.
\end{defn}

\begin{defn}
Let $X$ be a projective variety.  We say that a family of $k$-cycles $p: U \to W$ on $X$ is a rational family if every cycle-theoretic fiber lies in the same rational equivalence class.
\end{defn}

\begin{rmk}
Note that many intuitive constructions of families of cycles fail to meet the criteria: the map $\mathbb{A}^{2} \times \mathbb{A}^{2} \to \mathrm{Sym}^{2}\mathbb{A}^{2}$ is not flat over a characteristic $2$ field as pointed out in \cite{kollar96}.  Since we are primarily interested in the ``generic'' behavior of families of cycles, these shortcomings will not be important for us.
\end{rmk}

The following constructions show how to construct families of cycles from subsets $U \subset W \times X$.

\begin{constr}[Cycle version] \label{cycletofamilyconstr}
Let $X$ be a projective variety and let $W$ be an integral variety.  Suppose that $Z = \sum a_{i}V_{i}$ is a $(k+\dim W)$-cycle on $W \times X$ such that the first projection maps each $V_{i}$ dominantly onto $W$.  Let $W^{0} \subset W$ be the (non-empty) open locus over which every projection $p: V_{i} \to W$ is flat and let $U \subset \Supp(Z)$ denote the preimage of $W^{0}$.  Then the map $p: U \to W^{0}$ defines a family of cycles where we assign the coefficient $a_{i}$ to the component $V_{i} \cap U$ of $U$.
\end{constr}

\begin{constr}[Subscheme version] \label{equidimsubschemeconstr}
Suppose that $Y$ is a reduced variety and that $X$ is a projective variety.  Let $\tilde{U} \subset Y \times X$ be a closed subscheme such that the fibers of the projection $p: \tilde{U} \to Y$ are equidimensional of dimension $k$.  There is a natural way to construct a finite collection of families of effective cycles associated to the subscheme $\tilde{U}$.

Consider the image $p(\tilde{U})$ (with its reduced induced structure).  Let $\{ \tilde{W}_{j} \}$ denote the irreducible components of $p(\tilde{U})$.  For each there is a non-empty open subset $W_{j} \subset \tilde{W}_{j}$ such that the restriction of $p$ to each component of $p^{-1}(W_{j})_{red}$ is flat.  Since furthermore $p$ has equidimensional fibers, we obtain a family of effective $k$-cycles $p_{j}: U_{j} \to W_{j}$ where $U_{j} = p^{-1}(W_{j})_{red}$ and we assign coefficients so that the cycle underlying the family $p_{j}$ coincides with the fundamental cycle of $p^{-1}(W_{j})$.  We can then replace $\tilde{U}$ by the closed subscheme obtained by taking the base change to $p(\tilde{U}) - \cup_{j} W_{j}$ and repeat.  The end result is a collection of families $p_{i}: U_{i} \to W_{i}$ parametrizing the cycles contained in $\tilde{U}$.

If $p(\tilde{U})$ is irreducible and we are interested only in the generic behavior of the cycles in $\tilde{U}$, we can stop after the first step to obtain a single family of cycles.
\end{constr}

\subsection{Chow varieties and the Chow map}

Fix a projective variety $X$ and an ample divisor $H$ on $X$.  For any reduced scheme $Z$ over the ground field, \cite[Chapter I.3]{kollar96} introduces a more refined definition of a family of $k$-cycles of $X$ of $H$-degree $d$ over $Z$.  Koll\'ar then constructs a semi-normal projective variety $\Chow_{k,d,H}(X)$ that parametrizes families of effective $k$-cycles of $H$-degree $d$.  $\Chow(X)$ denotes the disjoint union over all $k$ and $d$ of $\Chow_{k,d,H}(X)$ for some fixed ample divisor $H$; it does not depend on the choice of $H$.

The precise way in which $\Chow(X)$ parametrizes cycles is somewhat subtle in characteristic $p$.  For a discussion of the Chow functor and universal families, see \cite{kollar96}.  We will need the following properties of $\Chow(X)$:
\begin{itemize}
\item Any family of cycles in the sense of Definition \ref{familydef} naturally yields a family of cycles in the refined sense of \cite[I.3.11 Definition]{kollar96} by applying \cite[I.3.14 Lemma]{kollar96} with the identity map (see also \cite[I.3.15 Corollary]{kollar96}).  
\item For any weakly normal integral variety $W$ and any (refined) family of effective cycles $p: U \to W$, there is an induced morphism $\ch_{p}: W \to \Chow(X)$ by \cite[I.4.8-I.4.10]{kollar96}.  (We will denote this map simply by $\ch$ when the family $p$ is clear from the context.)
\end{itemize}

For any family of effective cycles $p: U \to W$ the base change to the normal locus $W^{0} \subset W$ is still a family of cycles (where we assign the same coefficients).  Thus there is an induced rational map $\ch_{p}: W \dashrightarrow \Chow(X)$ that is a morphism on the normal locus of $W$.

The following crucial lemma encapsulates the set-theoretical nature of the Chow functors constructed in \cite[Chapter I.3]{kollar96}.

\begin{lem} \label{injectivechowlemma}
Let $X$ be a projective variety and let $p: U \to W$ be a family of effective $k$-cycles on $X$ over a weakly normal $W$.  A curve $C \subset W$ is contracted by $\ch: W \to \Chow(X)$ if and only if every cycle-theoretic fiber over $C$ has the same support.
\end{lem}

We will freely use the notation of \cite{kollar96} in the verification.

\begin{proof}
First suppose that every cycle-theoretic fiber of $p$ over $C$ has the same support.  Since $\Chow_{k,d,H}(X)$ is constructed by taking a semi-normalization (which is set-theoretically bijective), we may instead consider the induced map to $\Chow'_{k,d,H}(X)$.  This map factors through the map $\ch$ for a projective space containing an embedding of $X$; therefore it suffices to consider the case when $X = \mathbb{P}$.  Then the construction following \cite[Ch.~I Eq.~(3.23.1.5)]{kollar96} shows that the Cartier divisors on $(\mathbb{P}^{\vee})^{k+1}$ parametrized by the image of $C$ in $\mathbb{H}$ must all have the same support.  But this implies they are equal.

Conversely, suppose that $C$ is contracted by $\ch$.  As discussed in \cite[I.3.27.3]{kollar96}, $C$ is also contracted by the morphism to $\mathrm{Hilb}((\mathbb{P}^{\vee})^{k+1})$.  Again comparing with \cite[Ch.~I Eq.~(3.23.1.5)]{kollar96}, we see that the support of each of the cycles parametrized by $C$ is the same.
\end{proof}

It will often be helpful to replace a family $p: U \to W$ by a slightly modified version.

\begin{lem} \label{goodfamilymodification}
Let $X$ be a projective variety and let $p: U \to W$ be a family of effective cycles on $X$.  Then there is a normal projective variety $W'$ that is birational to $W$ and a family of cycles $p': U' \to W'$ such that $\ch(W') = \overline{\ch(W)}$.
\end{lem}

\begin{proof}
Let $\tilde{W}$ be any projective closure of $W$ and let $\tilde{U}$ be the closure of $U$ in $\tilde{W} \times X$.  Let $\phi: W' \to \tilde{W}$ be the normalization of a simultaneous flattening of the morphisms $\tilde{p}: \tilde{U}_{i} \to \tilde{W}$ for the components $\tilde{U}_{i}$ of $\tilde{U}$.  Let $U'$ denote the reduced subscheme of $W' \times X$ defined by the components of $\tilde{U} \times_{\tilde{W}} W'$ that dominate $W'$.  Since the components of $U'$ are in bijection with the components of $U$, we can assign to each component of $U'$ the coefficient of the corresponding component of $U$.  Then $p': U' \to W'$ is a family of effective $k$ cycles.  Since  the $p'$ and $p$ agree over an open normal subset of the base, the closure of the images under the map $\ch$ agree.
\end{proof}

\begin{rmk}
It is also important to know whether a rational family $p: U \to W$ can be extended to a \emph{rational} family over a projective closure of $W$ (although we will not need such statements below).  The arguments of \cite[Theorem 3]{samuel56} show that the subset of $\Chow(X)$ parametrizing cycles in a fixed rational equivalence class is a countable union of closed subvarieties.  Thus we can extend families in this way when working over an uncountable algebraically closed field $K$. 
\end{rmk}

\subsection{Chow dimension of families}

Let $p: U \to W$ be a family of effective $k$-cycles on a projective variety $X$.  Then all the cycle-theoretic fibers of $p$ are algebraically equivalent.  Indeed, for any two closed points of $W$, let $C$ be the normalization of a curve through those two points; since the base change of $U$ to $C$ is a union of flat families of effective cycles, we see that the corresponding cycle-theoretic fibers are algebraically equivalent.

\begin{defn}
Let $p: U \to W$ be a family of effective $k$-cycles on a projective variety $X$.  We say that $p$ represents $\alpha \in N_{k}(X)_{\mathbb{Z}}$ if the cycle-theoretic fibers of our family have class $\alpha$.
\end{defn}

\begin{defn}
Let $X$ be a projective variety and let $p: U \to W$ be an effective family of $k$-cycles.  We define
the Chow dimension of $p$ to be
\begin{equation*}
\chdim_{X}(p) := \dim(\overline{\mathrm{Im} \, \ch \!: W \dashrightarrow \Chow(X)})
\end{equation*}
If $\alpha \in N_{k}(X)_{\mathbb{Z}}$, we define
\begin{equation*}
\chdim_{X}(\alpha) = \max \{ \chdim(p) | p: U \to W \textrm{ represents } \alpha \}.
\end{equation*}
We will usually omit the subscript when it is clear from the context.
\end{defn}

\begin{rmk}
When our ground field $K$ has characteristic $0$, \cite{kollar96} constructs a universal family over any component of $\Chow(X)$.  Using Construction \ref{cycletofamilyconstr} this can be turned into a family of effective cycles in the sense of Definition \ref{familydef}.  Thus
\begin{equation*}
\chdim(\alpha) = \max \{ \dim(Y)\,  | \, Y \textrm{ is a component of } \Chow(X) \textrm{ representing  }\alpha \}.
\end{equation*}
Even when $K$ has characteristic $p$, for any component of $\Chow(X)$ \cite[I.4.14 Theorem]{kollar96} constructs a family of cycles whose chow map $\ch$ is dominant, so that we still have the same interpretation.
\end{rmk}

We can also consider the analogous construction for rational families.

\begin{defn}
Let $X$ be a projective variety.  For $\alpha \in N_{k}(X)_{\mathbb{Z}}$, we define
\begin{equation*}
\rchdim(\alpha) = \max \{ \chdim(p) | p: U \to W \textrm{ is a rational family representing } \alpha \}.
\end{equation*}
Similarly, for $\tau \in A_{k}(X)$ we define
\begin{equation*}
\rchdim(\tau) = \max \{ \chdim(p) | p: U \to W \textrm{ is a rational family of class } \tau \}.
\end{equation*}
\end{defn}

\subsection{Geometry of families}

\begin{defn}
Let $X$ be a projective variety and let $p: U \to W$ be a family of effective cycles on $X$.
\begin{itemize}
\item We say that $p$ is a reduced family if every coefficient $a_{i}$ is $1$.  Any family of effective cycles yields a reduced family by simply changing the coefficients.
\item We say that $p$ is an irreducible family if $U$ only has one component.
For any component $U_{i}$ of $U$, we have an associated irreducible family $p_{i}: U_{i} \to W$ (with coefficient $a_{i}$).
\end{itemize}
\end{defn}

\begin{lem} \label{genreduceddimest}
Let $X$ be a projective variety and let $p: U \to W$ be an effective family of $k$-cycles.  Let $p': U' \to W$ denote the reduced family for $p$.  Then $\chdim(p') = \chdim(p)$.
\end{lem}

\begin{proof}
Note that $p: U \to W$ and $p': U' \to W$ agree set-theoretically.  The statement then follows immediately from Lemma \ref{injectivechowlemma}.
\end{proof}

\begin{lem}  \label{componentsdimest}
Let $X$ be a projective variety and let $p: U \to W$ be a family of effective $k$-cycles on $X$.  For each component $U_{i}$ of $U$, let  $p_{i} := p|_{U_{i}}: U_{i} \to W$ denote the irreducible family induced by $U_{i}$.  Then $\chdim(p) \leq \sum_{i} \chdim(p_{i})$.
\end{lem}

\begin{proof}
By Lemma \ref{injectivechowlemma}, the map $\ch_{p}: W \dashrightarrow \Chow(X)$ factors rationally through the map $\prod_{i} \ch_{p_{i}}: W \dashrightarrow \prod_{i} \Chow(X)$.
\end{proof}

We will also need several geometric constructions.  

\begin{constr}[Flat pullback families] \label{flatpullbackfamilies}
Let $g: Y \to X$ be a flat morphism of projective varieties of relative dimension $d$.  Suppose that $p: U \to W$ is a family of effective $k$-cycles on $X$ with underlying cycle $V$.  The flat pullback cycle $g^{*}V$ on $W \times Y$ is effective and has relative dimension $(d+k)$ over $W$.   We define the flat pullback family $g^{*}p: U' \to W^{0}$ of effective $(d+k)$-cycles on $Y$ over an open subset $W^{0} \subset W$ by applying Construction \ref{cycletofamilyconstr} to $g^{*}V$.
\end{constr}

\begin{constr}[Pushforward families] \label{pushforwardoffamilies}
Let $f: X \to Y$ be a morphism of projective varieties.  Suppose that $p: U \to W$ is a family of effective $k$-cycles on $X$ with underlying cycle $V$.  Consider the cycle pushforward $f_{*}V$ on $W \times Y$ and assume $f_{*}V \neq 0$.  Construction \ref{cycletofamilyconstr} yields a family of $k$-cycles $f_{*}p: \tilde{U} \to W^{0}$ over an open subset of $W$.  We call $f_{*}p$ the pushforward family.  Note that this operation is compatible with the pushforward on cycle-theoretic fibers over $W^{0}$ by \cite[I.3.2 Proposition]{kollar96}.
\end{constr}

\begin{constr}[Restriction families] \label{restrictionfamilies}
Let $X$ be a projective variety and let $p: U \to W$ be a family of effective $k$-cycles on $X$.  Let $W' \subset W$ be an integral subvariety.  For each component $U_{i}$ of $U$, the restriction $U_{i} \times_{W} W'$ is flat over $W'$ of relative dimension $k$.  Consider the cycle $V$ on $W' \times X$ defined as the sum $V = \sum_{i} a_{i}V_{i}$ where $V_{i}$ is the fundamental cycle of $U_{i}$ restricted to $W'$.  We define the restriction of the family $p$ to $W'$ over an open subset $W'^{0} \subset W'$ by applying Construction \ref{cycletofamilyconstr} to $V$.  Note that this operation leaves the cycle-theoretic fibers unchanged over $W'^{0}$.  Note also that if $W' \subset W$ is open, then we may take $W^{0} = W'$ and the family $p$ is simply the base-change to $W^{0}$.
\end{constr}

\begin{constr}[Family sum] \label{familysumconstr}
Let $X$ be a projective variety and let $p: U \to W$ and $q: S \to T$ be two families of effective $k$-cycles on $X$.  We construct the family sum of $p$ and $q$ over an open subset of $W \times T$ as follows.  Let $V_{p}$ and $V_{q}$ denote the underlying cycles for $p$ and $q$ on $W \times X$ and $T \times X$ respectively.  The family sum of $p$ and $q$ is the family defined by applying Construction \ref{cycletofamilyconstr} to the sum of the flat pullbacks of $V_{p}$ and $V_{q}$ to $W \times T \times X$.
\end{constr}

\begin{constr}[Strict transform families]
Let $X$ be an integral projective variety and let $p: U \to W$ be a family of effective $k$-cycles on $X$.  Suppose that $\phi: X \dashrightarrow Y$ is a birational map.  We define the strict transform family of effective $k$-cycles on $Y$ as follows.

First, modify $U$ by removing all irreducible components whose image in $X$ is contained in the locus where $\phi$ is not an isomorphism.  Then define the cycle $U'$ on $W \times Y$ by taking the strict transform of the remaining components of $U$.  We define the strict transform family by applying Construction \ref{cycletofamilyconstr} to $U'$ over $W$.
\end{constr}

\begin{constr}[Intersecting against divisors] \label{intersectionfamilyconstr}
Let $X$ be a projective variety and let $p: U \to W$ be a family of effective $k$-cycles on $X$.  Let $D$ be an effective Cartier divisor on $X$.  If every cycle in our family has a component contained in $\Supp(D)$, we say that the intersection family of $p$ and $D$ is empty.

Otherwise, let $s: U \to X$ denote the projection map.  By assumption the effective Cartier divisor $s^{*}D$ does not contain any component of $U$, so we may take a cycle-theoretic intersection of $s^{*}D$ with the cycle underlying the family $p$ to obtain a $(k-1+\dim W)$-cycle $V$ on $W \times \Supp(D)$.  We then apply Construction \ref{cycletofamilyconstr} to obtain a family of cycles on $\Supp(D)$ over an open subset of $W$ which we denote by $p \cdot D$.  We can also consider the intersection as a family of cycles on $X$ by pushing forward.

Finally, suppose that we have a linear series $|L|$.  We define the intersection of $|L|$ with a family $p: U \to W$ as follows.  Consider the flat pullback family $q: U' \to W^{0}$ on $\mathbb{P}(|L|) \times X$.  Then intersect the family $q$ against the pullback of the universal divisor on $\mathbb{P}(|L|) \times X$ to obtain a family of cycles on $\mathbb{P}(|L|) \times X$.  The underlying cycle has dimension $k-1+\dim W + \dim(\mathbb{P}(|L|))$; by using Construction \ref{cycletofamilyconstr}, we can convert this cycle to a family of effective $(k-1)$-cycles on $X$ over an open subset of $W \times \mathbb{P}(|L|)$.
\end{constr}

We conclude this section with a brief analysis of how these constructions affect the Chow dimension.

\begin{lem} \label{genfinitepushlem}
Let $f: X \to Y$ be a morphism of projective varieties.  Let $p: U \to W$ be a family of effective $k$-cycles on $X$ such that for every component $U_{i}$ of $U$ the image $\overline{s(U_{i})}$ is not contracted to a variety of smaller dimension by $f$.  Then $\chdim(p) = \chdim(f_{*}p)$.
\end{lem}

\begin{proof}
Let $T$ be an integral curve through a general point of $W$ that is not contracted by $\ch_{p}$ and set $S = p^{-1}(T)$.  Then $\dim(\overline{s(S)}) = \dim(f(\overline{s(S)}))$.  Thus the cycle-theoretic fibers parametrized by $T$ do not pushforward to the same cycle on $Y$.  We conclude by Lemma \ref{injectivechowlemma} that $T$ is not contracted by $\ch_{f_{*}p}$.
\end{proof}

\begin{lem} \label{chdimfamilysumlem}
Let $X$ be a projective variety and let $p: U \to W$ and $q: S \to T$ be two families of effective $k$-cycles on $X$.  Then $\chdim(p+q) = \chdim(p) + \chdim(q)$.
\end{lem}

\begin{proof}
A curve through a general point of $(W \times T)^{0}$ is contracted by $\ch_{p+q}$ if and only if its projection to $W$ and to $T$ are contracted by $\ch_{p}$ and $\ch_{q}$ respectively.  We conclude by Lemma \ref{injectivechowlemma}.
\end{proof}

\section{Chow dimension of projective space} \label{chowdimpnsec}

In this section we compute the dimension of the Chow variety of $\mathbb{P}^{n}$.  \cite{eh92} computes the dimension of the Chow variety of curves.  The precise statement is as follows: let $\ell$ denote the class of a line on $\mathbb{P}^{n}$.  Then for $d>1$,
\begin{equation*}
\dim \Chow(\mathbb{P}^{n},d \ell) = \max \left\{ 2d(n-1), \frac{d^{2} + 3d}{2} + 3(n-2) \right\}.
\end{equation*}
Note that the first number is the dimension of the space of unions of $d$ lines on $\mathbb{P}^{n}$, and the second is the dimension of the space of degree $d$ plane curves.

\begin{rmk}
Although \cite{eh92} does not specify the ground field, some of the references explicitly work only over $\mathbb{C}$.  However, the results of \cite{eh92} holds equally well over any algebraically closed field.  The argument of \cite{eh92} involves mainly estimates on the dimension of the space of sections of a normal sheaf.  Since we are only working with embedded curves, we do not need to worry about pathologies of tangent spaces in characteristic $p$.  The only additional verifications one needs to make are:
\begin{itemize}
\item The Halphen bounds on genus, and their extensions in \cite{harris82}, hold in arbitrary characteristic using the same arguments.
\item The deformation theoretic results of \cite{ac81} used in the paper are also true in arbitrary characteristic.  Indeed, suppose that $f: C \to \mathbb{P}^{n}$ is a morphism from a smooth curve such that $f$ restricts to an isomorphism from an open set of $C$ onto its image.  Using the deformation theory for maps as explained in \cite{sernesi06}, one sees that deformations of $f$ that change the image must be associated with first-order deformations which are not torsion sections of the normal sheaf (since they can not fix an open subset).
\end{itemize}
\end{rmk}

In this section we prove the analogue of \cite[Theorem 3]{eh92} in arbitrary dimension.
\begin{thrm} \label{chowdimpn}
Let $\alpha$ denote the class of the $k$-plane on $\mathbb{P}^{n}$.  Then for $d>1$ the dimension of $\Chow(\mathbb{P}^{n},d \alpha)$ is
\begin{align*}
\max \left\{ d(k+1)(n-k), \left( \begin{array}{c} d+k+1 \\ k+1 \end{array} \right) - 1 + (k+2)(n-k-1) \right\}.
\end{align*}
\end{thrm}
Again, the first number is the dimension of the space of a union of $d$ $k$-planes on $\mathbb{P}^{n}$ and the second is the dimension of the space of degree $d$ hypersurfaces in $(k+1)$-planes.

The basic tool in \cite{eh92} is Castelnuovo theory for curves.  It is unclear how best to formulate an analogue in higher dimension.  Thus we take an alternative approach: the strategy is to reduce the general case to the calculation for curves using hyperplane sections.  By induction, one can find a bound on the dimension of the intersection family.  The main step is to estimate the dimension of the ``kernel'', in other words, the dimension of a family of $k$-cycles which has fixed intersection with a hyperplane.  The following lemmas build up to such a statement.

\begin{lem} \label{irreducibilitylem}
Let $R$ be a normal integral projective variety of dimension $n \geq 1$.  Suppose that the Cartier divisor $A$ on $R$ is the pullback of a very ample divisor under a birational map and set $d = A^{n}$.  Define
\begin{equation*}
S = \mathbb{P}_{R}(\mathcal{O}_{R} \oplus \mathcal{O}_{R}(A))
\end{equation*}
and $\pi: S \to R$ the projection.  Set $L$ to be a Cartier divisor representing the relative dualizing sheaf $\mathcal{O}_{S/R}(1)$ and $E$ to be the effective Cartier divisor corresponding to the unique section of $\mathcal{O}_{S/R}(1) \otimes \pi^{*}\mathcal{O}_{R}(-A)$.  Then any irreducible effective Weil divisor $D$ satisfying $L^{n} \cdot D = d$ and $E \cdot D = 0$ must lie in $|L|$ (and in particular must be Cartier).
\end{lem}

\begin{proof} Note that for any rank one reflexive sheaf $\mathcal{L}$ on $R$ we have that $\pi^{*}\mathcal{L}$ is still reflexive.  Furthermore, $\pi_{*}\pi^{*}\mathcal{L} \cong \mathcal{L}$, since upon restriction to the smooth locus $U \subset R$ both agree with $\pi_{*}\pi^{*}\mathcal{L}|_{U}$.

Consider the reflexive sheaf $\mathcal{O}_{S}(D)$.  There is a unique integer $q$ and reflexive rank one sheaf $\mathcal{O}_{R}(T)$ such that
\begin{equation*}
\mathcal{O}_{S}(D) \cong \mathcal{O}_{S/R}(1)^{\otimes q} \otimes \pi^{*}\mathcal{O}_{R}(T).
\end{equation*}
(One can verify this over the smooth locus of $R$ using the usual description of the Picard group of a projective bundle.)  The numerical conditions on $D$ imply that $q=1$ and $[T] \in N_{n-1}(R)$ is the $0$ class.  Then by pushing forward we see that
\begin{equation*}
H^{0}(S,\mathcal{O}_{S}(D)) \cong H^{0}(R,\mathcal{O}_{R}(T+A) \oplus \mathcal{O}_{R}(T)).
\end{equation*}
Suppose that $\mathcal{O}_{R}(T)$ is not the trivial bundle.  Then $H^{0}(R,\mathcal{O}_{R}(T)) = 0$ and sections of $\mathcal{O}_{S}(D)$ are constructed by taking a flat pullback of an effective divisor in $H^{0}(R,\mathcal{O}_{R}(T+A))$ and adding on $E$.  Such divisors can not be irreducible.  Thus $T=0$ and we have proved the statement.
\end{proof}

\begin{lem} \label{dimofrestrictedfam}
Let $p: U \to W$ be a family of irreducible degree $d$ effective $k$-cycles on $\mathbb{P}^{n}$ for $1 \leq k < n$.  Suppose there is a fixed reduced subvariety $Z$ of degree $d$ and dimension $k-1$ such that every member of the family contains $Z$.  Then 
\begin{equation*}
\chdim(p) \leq (n-k-1) + \left( \begin{array}{c} d+k \\ k+1 \end{array} \right).
\end{equation*}
\end{lem}

\begin{proof}
The proof is by decreasing induction on the codimension $n-k$.  For the base case, suppose first that $p$ consists of a family of divisors.  Then we find that the family has dimension at most $h^{0}(\mathbb{P}^{n},\mathcal{I}_{Z}(d))$.  We can find an upper bound on this dimension by considering a toric deformation of $Z$ to a degenerate subscheme $Z'$ contained in a hyperplane $H$.  Note that for a general member $Z^{*}$ of this deformation the dimension $h^{0}(\mathbb{P}^{n},\mathcal{I}_{Z^{*}}(d))$ coincides with that for $Z$.  Then by upper semicontinuity
\begin{align*}
h^{0}(\mathbb{P}^{n},\mathcal{I}_{Z}(d)) & \leq h^{0}(\mathbb{P}^{n},\mathcal{I}_{Z'}(d)) \\
&  = \dim \ker(H^{0}(\mathbb{P}^{n},\mathcal{O}(d)) \to H^{0}(H,\mathcal{O}_{H}(d)))
\end{align*}
and we obtain the desired inequality for $n - k=1$.  

Now suppose $n-k>1$.  For simplicity, we may assume that $W$ is actually a subvariety of $\Chow(\mathbb{P}^{n})$.  Consider the projection away from a general point in $\mathbb{P}^{n}$ and let $\phi: P \to \mathbb{P}^{n}$ and $\pi: P \to \mathbb{P}^{n-1}$ be a resolution.  We may assume a general element of $p$ does not contain the center of the projection.  Note that the pushforward family of $p$ again satisfies the hypotheses of the theorem as a family on $\mathbb{P}^{n-1}$.  Thus we seek an upper bound on $\dim(F)$ where $F$ is a general fiber of the rational map $\ch_{\pi_{*}p}: W \dashrightarrow \Chow(\mathbb{P}^{n-1})$.

Fix a general cycle $T$ in our family so that $\pi(T)$ is irreducible of degree $d$ and $F$ is the fiber corresponding to the point $[T] \in \Chow(\mathbb{P}^{n-1})$.  Let $\nu: R \to \pi(T)$ be the normalization of $\pi(T)$, and consider the bundle
\begin{equation*}
S := \mathcal{O}_{R} \oplus  \nu^{*}\mathcal{O}(1).
\end{equation*}
$S$ naturally maps to $\pi^{-1}\pi(T)$, and hence also maps naturally $\psi: S \to \mathbb{P}^{n}$.  We let $L = \psi^{*}H|_{S}$ for a general hyperplane $H$ and let $E$ denote the pull-back to $S$ of the exceptional divisor for $\phi$.  Let $p'$ denote the strict transform to $S$ of the subfamily of $p$ defined by $F$.  First note that any element $D$ of $p'$ satisfies $L^{k} \cdot D = d$ and $E \cdot D = 0$.  Thus any element of $p'$ must actually lie in $|L|$ by Lemma \ref{irreducibilitylem}.

Now fix a general element $D$ of $p'$.  By degree considerations, the restriction of every element of $p'$ to $D$ coincides with $Z$.  In other words, $p'$ lies in a fiber of the restriction map $|L| \to |L|_{D}|$.  Taking sections of the short exact sequence
\begin{equation*}
0 \to \mathcal{O}_{S} \to \mathcal{O}_{S}(L) \to \mathcal{O}_{D}(L) \to 0
\end{equation*}
we see that $\chdim(p') \leq 1$.  Then, arguing by induction we conclude
\begin{equation*}
\chdim(p) \leq \chdim(\pi_{*}p) + 1 \leq (n-k-2) + \left( \begin{array}{c} d+k \\ k+1 \end{array} \right) + 1.
\end{equation*}
\end{proof}

\begin{thrm}
Let $p: U \to W$ be an irreducible reduced family of degree $d$ effective $k$-cycles on $\mathbb{P}^{n}$.  Then
\begin{equation*}
\chdim(p) \leq \left( \begin{array}{c} d+k+1 \\ k+1 \end{array} \right) - 1 + (k+2)(n-k-1).
\end{equation*}
\end{thrm}

\begin{proof}
The proof is by induction on $k$.  When $k=1$, we have a family of curves on $\mathbb{P}^{n}$ and the statement is proved by \cite{eh92}.

For arbitrary $k$, fix a general hyperplane $H$.  Consider the intersection family $p \cdot H$.  By Lemma \ref{dimensionlemma}, we have
\begin{equation*}
\chdim_{\mathbb{P}^{n}}(p) \leq \chdim_{\mathbb{P}^{n}}(p|_{F}) + \chdim_{H}(p \cdot H)
\end{equation*}
where $F$ is a general fiber of $\ch_{p \cdot H}: W \dashrightarrow \Chow(H)$.  However, Lemma \ref{dimofrestrictedfam} shows that
\begin{equation*}
\chdim_{\mathbb{P}^{n}}(p|_{F}) \leq (n-k-1) + \left( \begin{array}{c} d+k \\ k+1 \end{array} \right)
\end{equation*}
and since $p \cdot H$ is a family of irreducible $(k-1)$-cycles in $\mathbb{P}^{n-1}$, by induction
\begin{equation*}
\chdim_{H}(p \cdot H) \leq \left( \begin{array}{c} d+k \\ k \end{array} \right) - 1 + (k+1)(n-k-1).
\end{equation*}
Combining these two bounds gives the result.
\end{proof}

\begin{proof}[Proof of Theorem \ref{chowdimpn}:]
Note that as $\lambda := (a_{1},\ldots,a_{q})$ varies over all partitions of $d$, we have
\begin{equation*}
\dim \Chow(\mathbb{P}^{n},d \alpha) = \sup_{\lambda} \left\{ \sum_{i=1}^{q} \dim \Chow_{irr}(\mathbb{P}^{n},a_{i} \alpha) \right\}
\end{equation*}
where $\Chow_{irr}$ denotes the components of $\Chow$ parametrizing irreducible reduced subvarieties.

For simplicity, we define the functions $f(d) = d(k+1)(n-k)$ and
\begin{equation*}
g(d) =  \left( \begin{array}{c} d+k+1 \\ k+1 \end{array} \right) - 1 + (k+2)(n-k-1).
\end{equation*}
Let $r$ be the largest integer such that $f(r) \geq g(r)$.  For $d \leq r$, using the expression above one sees immediately that the dimension of $\Chow(\mathbb{P}^{n},d\alpha)$ is $d(k+1)(n-k)$.

Now suppose $d>r$.  By induction on $d$ to show the theorem it suffices to prove that
\begin{equation*}
g(d) \geq \sup_{1 \leq q \leq d-1} \{ g(q) + g(d-q)\} \qquad \textrm{and} \qquad g(d) \geq g(d-1) + f(1).
\end{equation*}
(Note that it suffices to restrict our attention to $f(1)$ since by induction for any $d'<d$ a degree $d'$ family whose dimension agrees with $f(d')$ can be replaced by the sum of a degree $d'-1$ family and a degree $1$ family.)  Since $d>r$ we must have
\begin{equation*}
\left( \begin{array}{c} d+k \\ k \end{array} \right) \geq  \left( \begin{array}{c} r+1+k \\ k \end{array} \right).\end{equation*}
Since we must have $f(r+1)-f(r) > g(r+1)-g(r) = (k+1)(n-k)$, the inequality $g(d) \geq g(d-1)+f(1)$ follows from the previous equation.  In fact, this completes the proof for $d \leq 2r$.

For $d > 2r$, we still need to check the other inequality $g(d) \geq g(q) + g(d-q)$ where $1 \leq q \leq d-1$.  In this degree range we may assume without loss of generality that $d-q > r$.  Then note that
\begin{align*}
\left( \begin{array}{c} d+k+1 \\ k+1 \end{array} \right) & - \left( \begin{array}{c} d-q+k+1 \\ k+1 \end{array} \right) - \left( \begin{array}{c} q+k+1 \\ k+1 \end{array} \right)  \\
& = \sum_{i=0}^{q} \left( \left( \begin{array}{c} d-q+k+i \\ k \end{array} \right) - \left( \begin{array}{c} k+i \\ k \end{array}  \right) \right) \\
& \geq (q+1) \left( \left( \begin{array}{c} d-q+k \\ k \end{array}  \right) - 1 \right) \\
& \geq (q+1) \left( \left( \begin{array}{c} r+1+k \\ k \end{array}  \right) - 1 \right) \\
& \geq (q+1)((k+1)(n-k) - 1) \\
& \geq 2(k+1)(n-k) - 2 \\
& > (k+2)(n-k-1) - 1
\end{align*}
and the conclusion follows.
\end{proof}

\section{The variation function} \label{variationsection}

The variation of a class $\alpha \in N_{k}(X)_{\mathbb{Z}}$ measures the rate of growth of the dimensions of components of $\Chow(X)$ that represent $m\alpha$ as $m$ increases.  The main theorem in this section is Theorem \ref{variationandbigness} which shows that variation is in some sense a measure of bigness  along subvarieties of $X$.

\subsection{Dimensions of families of cycles}

Before defining the variation, we need to find bounds for the dimension of components of $\Chow(X)$.  The following theorem incorporates a suggestion of Voisin who pointed out that the coefficient in the original version could be improved by considering a generically finite map to projective space.

\begin{thrm} \label{familydim}
Let $X$ be an equidimensional projective variety of dimension $n$ and let $\alpha \in N_{k}(X)_{\mathbb{Z}}$.  Suppose that $A$ is a very ample divisor on $X$ and set $d = \alpha \cdot A^{k}$.  Then we have
\begin{equation*}
\chdim(\alpha) \leq  \left( \begin{array}{c} d+k+1 \\ k+1 \end{array} \right) + d(k+1)(n-k).
\end{equation*}
\end{thrm}

\begin{proof}
Suppose that $p: U \to W$ is a family of effective divisors representing $\alpha$.  Since the desired upper bound is superadditive in $d$, by Lemma \ref{componentsdimest} we may prove the bound for each irreducible component of $U$ separately.  Hence we assume that $p$ is an irreducible family.

Let $\pi: X \to \mathbb{P}^{n}$ be a generically finite morphism defined by a general subspace of $|A|$.  There are two possibilities:
\begin{enumerate}
\item The general member of the family $p$ is contracted by $\pi$.  Then every member of $p$ must be contained in the $\pi$-exceptional locus of $X$, and we conclude by induction on the dimension $n$.
\item The general member of $p$ is not contracted by $\pi$.  By Lemma \ref{genfinitepushlem} we see that $\chdim(p) = \chdim(\pi_{*}p)$.  Thus we obtain the upper bound by Theorem \ref{chowdimpn}.
\end{enumerate}
\end{proof}

Theorem \ref{familydim} shows that for any class $\alpha \in N_{k}(X)_{\mathbb{Z}}$, there is some positive constant $C$ such that $\chdim(m\alpha) < Cm^{k+1}$.  Furthermore, this order of growth can always be achieved by some class on $X$ as in the following example.

\begin{exmple} \label{bigvariationexample}
Let $H_{1},\ldots,H_{n-k}$ be general very ample divisors on $X$ and set $\alpha = H_{1} \cdot \ldots \cdot H_{n-k}$.  Let $V$ denote the scheme-theoretic intersection $H_{1} \cap \ldots \cap H_{n-k-1}$.  The linear series $|m(H_{n-k}|_{V})|$ defines a rational family of $k$-cycles representing $m\alpha$.  Thus
\begin{equation*}
\rchdim(m\alpha) \geq (H_{1} \cdot \ldots \cdot H_{n-k-1} \cdot H_{n-k}^{k+1})\frac{m^{k+1}}{(k+1)!} + O(m^{k})
\end{equation*}
\end{exmple}

\subsection{Definitions}

Theorem \ref{familydim} and Example \ref{bigvariationexample} suggest that one should compare the growth rate of $\chdim(m\alpha)$ against $m^{k+1}$.  

\begin{defn} \label{variationdefn}
Let $X$ be an integral projective variety.  For any $\alpha \in N_{k}(X)_{\mathbb{Z}}$, we define the variation of $\alpha$ to be
\begin{equation*}
\var(\alpha) := \limsup_{m \to \infty} \frac{\chdim(m\alpha)}{m^{k+1}/(k+1)!}.
\end{equation*}
We define the rational variation of $\alpha$ in the analogous way using $\rchdim(\alpha)$ in place of $\chdim(\alpha)$.
\end{defn}

The choice of the coefficient $(k+1)!$ is justified by Example \ref{cyclesinpn}.

\begin{rmk}
\cite[Corollary 1.6]{nollet97} shows that there are infinitely many components of $\Hilb(\mathbb{P}^{3})$ parametrizing subschemes whose underlying cycle is a double line.  Furthermore the dimensions of these components is unbounded.  Thus there is no analogue of Theorem \ref{familydim} for components of the Hilbert scheme with a fixed underlying cycle class.

Any attempt to formulate an analogue of the variation using the Hilbert scheme will need to either consider only special components of $\Hilb(X)$ or account for all the terms in the Hilbert polynomial.
\end{rmk}

\begin{exmple} \label{cyclesinpn}
Let $\alpha \in N_{k}(\mathbb{P}^{n})$ denote the class of a $k$-dimensional hyperplane.  Theorem \ref{chowdimpn} shows that for sufficiently large degrees $d$, there is a component of $\Chow(\mathbb{P}^{n},d\alpha)$ of maximal dimension parametrizing degree $d$ hypersurfaces in a $(k+1)$-dimensional hyperplane.  Thus we have $\var(\alpha)=1$.
\end{exmple}

\begin{exmple} \label{variationofdivisors}
Let $X$ be a normal integral projective variety of dimension $n$ and suppose that $X$ admits a resolution $\phi: Y \to X$ such that the kernel of $\phi_{*}: N_{n-1}(Y) \to N_{n-1}(X)$ is spanned by $\phi$-exceptional divisors.  (For example $X$ could be smooth or normal $\mathbb{Q}$-factorial over $\mathbb{C}$.)  Then for any Cartier divisor $D$ on $X$ we have $\ratvar([D]) = \var([D]) = \vol(D)$. 

To verify this, note that since $\chdim$ is preserved by passing to strict transform families of divisors we have
\begin{equation*}
\chdim(m[D]) = \max \left\{ \chdim(\beta) \left| \begin{array}{l} \beta \textrm{ is a class on } Y \\ \textrm{with } \phi_{*}\beta = m[D] \end{array} \right. \right\}.
\end{equation*}
Let $L$ denote any divisor in the class $\beta$ attaining this maximum value.  We may write $L \equiv m\phi^{*}D + E$ where $E$ is some $\phi$-exceptional divisor.  Since increasing the coefficients in $E$ can only increase $\chdim$, we may assume that $E$ is effective.  But then
\begin{align*}
h^{0}(Y,\mathcal{O}_{Y}(L)) = h^{0}(Y,\mathcal{O}_{Y}(L - E))
\end{align*}
by the negativity of contraction lemma (see for example \cite[III.5.7 Proposition]{nakayama04}).
Thus
\begin{align*}
h^{0}(X,\mathcal{O}_{X}&(mD))-1 \leq \rchdim(m[D]) \leq \chdim(m[D]) \\
&  \leq \dim \Pic^{0}(Y) + \max_{D' \equiv m\phi^{*}D} h^{0}(Y,\mathcal{O}_{Y}(D')) - 1.
\end{align*}
While the rightmost term may be greater than $h^{0}(X,\mathcal{O}_{X}(mD))-1$, the difference is bounded by a polynomial of degree $n-1$ in $m$ (see the proof of \cite[Proposition 2.2.43]{lazarsfeld04}).  Thus $\ratvar([D])$ and $\var([D])$ agree with the volume.
\end{exmple}

\begin{exmple} \label{0cyclevol}
Suppose that $X$ is an integral projective variety of dimension $n$.  There is an isomorphism $\deg \!: N_{0}(X)_{\mathbb{Z}} \to \mathbb{Z}$.  For $\alpha \in N_{0}(X)_{\mathbb{Z}}$ of positive degree $\chdim(\alpha) = n\deg(\alpha)$ so that
\begin{equation*}
\var(\alpha) = n\deg(\alpha)
\end{equation*}

The behavior of $\ratvar(\alpha)$ is somewhat more subtle.  For simplicity suppose that $\alpha$ is the positive generator of $N_{0}(X)_{\mathbb{Z}}$.  A result of \cite{roitman72} shows that there are non-negative integers $d(X), j(X)$ such that for sufficiently large $m$
\begin{equation*}
\rchdim(m\alpha) \geq m(n - d(X)) - j(X).
\end{equation*} 
This gives a lower bound on $\ratvar$.  However, since $\ratvar$ calculates the maximal variation (and not the ``minimum variation'' as in \cite{roitman72}), it may happen that $\ratvar(\alpha) > n-d(X)$.  At the very least we know that $\ratvar(\alpha)$ is always positive.  In fact, by considering families of points lying on a fixed curve $C \subset X$ we see that $\ratvar(\alpha) \geq 1$.
\end{exmple}

\begin{rmk}
The components of $\Chow(X)$ of maximal dimension tend to parametrize degenerate subvarieties.  To measure the bigness of a class, one should instead only consider the dimensions of components of $\Chow(X)$ that are ``general'' in some sense.  This intuition is captured by the mobility function defined in \cite{lehmann13}; however, it would be interesting to see a formulation using the Chow variety directly.  For curves on $\mathbb{P}^{3}$, \cite{perrin87} conjectures that calculating the dimensions of components of $\Chow(X)$ parametrizing ``general'' curves of degree $d$ -- in the sense that the corresponding cycles are not contained in any hypersurface of degree $<d^{1/2}$ -- will yield the value of the mobility function.
\end{rmk}

\subsection{Basic properties}

We next verify some of the basic properties of the variation.

\begin{lem} \label{rescalingvariation}
Let $X$ be an integral projective variety and let $\alpha \in N_{k}(X)_{\mathbb{Z}}$.  Then for any positive integer $c$ we have $\var(c\alpha) = c^{k+1}\var(\alpha)$ (and similarly for $\ratvar$).
\end{lem}

\begin{proof}
If $\chdim(\alpha) > 0$ then $\alpha$ is represented by an effective cycle $Z$.  Thus $\chdim(\alpha + \beta) \geq \chdim(\beta)$ for any class $\beta$: if $p$ is a family of effective cycles of class $\beta$ then we can add the constant cycle $Z$ to $p$ (using the family sum construction) to obtain a family representing $\alpha + \beta$ with the same Chow dimension.  We conclude by the following Lemma \ref{lazlemma}.
\end{proof}

\begin{lem}[\cite{lazarsfeld04} Lemma 2.2.38] \label{lazlemma}
Let $f: \mathbb{N} \to \mathbb{R}_{\geq 0}$ be a function.  Suppose that for any $r,s \in \mathbb{N}$ with $f(r) > 0$ we have that $f(r+s) \geq f(s)$.
Then for any $k \in \mathbb{R}_{>0}$ the function $g: \mathbb{N} \to \mathbb{R} \cup \{ \infty \}$ defined by
\begin{equation*}
g(r) := \limsup_{m \to \infty} \frac{f(mr)}{m^{k}}
\end{equation*}
satisfies $g(cr) = c^{k}g(r)$ for any $c,r \in \mathbb{N}$.
\end{lem}

\begin{rmk}
Although \cite[Lemma 2.2.38]{lazarsfeld04} only explicitly address the volume function, the essential content of the proof is the more general statement above.
\end{rmk}

Lemma \ref{rescalingvariation} allows us to extend the definition of variation to any $\mathbb{Q}$-class by homogeneity.  Thus we obtain a function
\begin{equation*}
\var: N_{k}(X)_{\mathbb{Q}} \to \mathbb{R}_{\geq 0}.
\end{equation*}


\begin{lem} \label{variationincreases}
Let $X$ be an integral projective variety.  Suppose that $\alpha, \beta \in N_{k}(X)_{\mathbb{Q}}$ are classes such that some positive multiple of each is represented by an effective cycle.  Then $\var(\alpha + \beta) \geq \var(\alpha) + \var(\beta)$ (and similarly for $\ratvar$).
\end{lem}

\begin{proof}
Note that we may check the inequality after rescaling $\alpha$ and $\beta$ by the same positive integer $c$.  Thus we may suppose that every multiple of $\alpha$ and $\beta$ is represented by an effective cycle.

Suppose that $p: U \to W$ is a family representing $m\alpha$ and $q: S \to T$ is a family representing $m\beta$.  Then the family sum $p+q$ represents $m(\alpha+\beta)$.  Lemma \ref{chdimfamilysumlem} shows that
\begin{equation*}
\chdim(p+q) = \chdim(p) + \chdim(q)
\end{equation*}
and the desired inequality follows.
\end{proof}

By Example \ref{bigvariationexample}, we find:

\begin{cor} \label{bigvarcor}
Let $X$ be an integral projective variety and let $\alpha \in N_{k}(X)_{\mathbb{Q}}$ be a big class.  Then $\var(\alpha) > 0$.
\end{cor}

As a consequence, we see that $\var$ is a continuous function on the big cone.

\begin{thrm} \label{bigvarcont}
Let $X$ be an integral projective variety.  The function $\var: N_{k}(X)_{\mathbb{Q}} \to \mathbb{R}_{\geq 0}$ is locally uniformly continuous on the interior of $\Eff_{k}(X)_{\mathbb{Q}}$ (and similarly for $\ratvar$).
\end{thrm}

\begin{proof}
$\var$ verifies conditions (1)-(3) of the following Lemma \ref{easyconelem}.
\end{proof}

\begin{lem}[\cite{lehmann13} Lemma 2.8] \label{easyconelem}
Let $V$ be a finite dimensional $\mathbb{Q}$-vector space and let $C \subset V$ be a salient full-dimensional closed convex cone.  Suppose that $f: V \to \mathbb{R}_{\geq 0}$ is a function satisfying
\begin{enumerate}
\item $f(e) > 0$ for any $e \in C^{int}$,
\item there is some constant $c > 0$ so that $f(me) = m^{c}f(e)$ for any $m \in \mathbb{Q}_{>0}$ and $e \in C$, and
\item for every $v \in C$ and $e \in C^{int}$ we have $f(v+e) \geq f(v)$.
\end{enumerate}
Then $f$ is locally uniformly continuous on $C^{int}$.
\end{lem}

The behavior of the variation along the pseudo-effective boundary is more subtle.  Probably the most one can hope for is:

\begin{ques}
Let $X$ be an integral projective variety.  Is the function $\var: \Eff_{k}(X)_{\mathbb{Q}} \to \mathbb{R}_{\geq 0}$ upper semi-continuous?
\end{ques}

By analogy with the volume, the variation should satisfy some form of concavity on the big cone.  The following conjecture is a weak statement in this direction.  It is easy to show that the conjecture would imply the upper semi-continuity of $\var$ as a function on $\Eff_{k}(X)_{\mathbb{Q}}$.

\begin{conj} \label{volconvexconj}
Let $X$ be an integral projective variety.  Suppose that $\alpha,\beta,\gamma \in N_{k}(X)_{\mathbb{Z}}$ are classes with $\alpha$ pseudo-effective and $\beta$ and $\gamma$ big.  Then
\begin{equation*}
\chdim(\alpha + \beta) - \chdim(\alpha) \leq \chdim(\alpha + \beta + \gamma) - \chdim(\alpha + \gamma).
\end{equation*}
\end{conj}

Finally we note that variation behaves well with respect to inclusions of subvarieties.

\begin{lem} \label{variationinclusionlem}
Let $X$ be an integral projective variety and $i: W \to X$ an integral closed subvariety.  For any class $\beta \in N_{k}(W)_{\mathbb{Q}}$ we have $\var(\beta) \leq \var(i_{*}\beta)$ (and similarly for $\ratvar$).
\end{lem}

\begin{proof}
Let $p$ be a family of effective cycles on $W$ and consider the push-forward family $q$ on $X$.  Recall that for a general cycle-theoretic fiber $Z$ of $p$ the corresponding cycle in the push-forward family is just $i_{*}Z$; thus Lemma \ref{injectivechowlemma} shows that $\chdim(p) = \chdim(q)$ and the result follows.
\end{proof}

\subsection{Variation, connecting chains, and bigness} \label{connectionsandbignesssection}

Example \ref{p3blowup} shows that a class may have positive variation even when it is not big.  This class is constructed by pushing forward a big class on a subvariety.  In this section we show that every class with positive variation arises in this way.  

The main step in the proof is to develop a criterion for bigness of a class using connecting chains of cycles.  This criterion is modeled on \cite[Theorem 2.4]{8authors} which describes big curve classes via connecting chains.  The correct analogue in higher dimensions should require that the cycles in our chain intersect ``positively'' in some sense.  The next theorem shows that such a statement holds under a very strong positivity condition.

\begin{defn} \label{stronglyampleconnectingdefn}
Let $X$ be an integral projective variety of dimension $n$ and let $p: U \to W$ be a family of effective $k$-cycles.  We say that $p$ is strongly big-connecting if $s: U \to X$ is dominant and there is a big effective divisor $B$ on $X$ such that every $p$-horizontal component of $s^{*}B$ is contracted to a subvariety of $X$ of dimension at most $k-1$.
\end{defn}

The following lemma is in preparation for Theorem \ref{stronglyampleisbig}.

\begin{lem} \label{stramphelp}
Let $p: U \to W$ be a flat map of projective varieties of relative dimension $k$ with $W$ integral.  Let $A$ be a big effective Cartier divisor on $U$.  There is a big effective $k$-cycle $Z$ on $U$ such that every component of $\Supp(Z)$ is contained either in a fiber of $p$ or in a $p$-horizontal component of $A$.
\end{lem}

\begin{proof}
Set $n=\dim U$.  The proof is by induction on the codimension $n-k$.  For the base case $n-k=1$ we can simply take $Z=A$.

For the general case, note that by Lemma \ref{intlem} there is a big effective $k$-cycle $V$ on $U$ contained in $\Supp(A)$.  Let $D$ denote the support of the $p$-vertical components of $A$ and let $V'$ denote the part of $V$ whose support is not contained in any $p$-horizontal component of $A$.

Choose a very ample divisor $H$ on $W$ sufficiently positive so that $p^{*}H - D$ is numerically equivalent to an effective divisor.  For $H$ general, and hence integral, we can apply the induction hypothesis to $p: p^{*}H \to H$ and $A|_{p^{*}H}$ to obtain a big effective $k$-cycle $Z'$ on $p^{*}H$ satisfying the support condition.  In particular, for an ample divisor $\tilde{H}$ on $U$ there is some $c$ sufficiently large so that $c[Z'] \succeq p^{*}H \cdot \tilde{H}^{n-k-1}$.  This shows that some multiple of $[Z']$ will also dominate any effective cycle supported in $D$, so for some $c'$ we have $c'[Z'] \succeq [V']$.   Set $Z = c'Z' + (V-V')$.
\end{proof}

\begin{thrm} \label{stronglyampleisbig}
Let $X$ be an integral projective variety.  A class $\alpha \in N_{k}(X)$ is big if and only if there is some strongly big-connecting family of effective $k$-cycles $p: U \to W$ with class $\beta$ and a positive constant $c$ such that $c\alpha - \beta$ is pseudo-effective.
\end{thrm}

\begin{proof}
We first prove the forward implication.  It suffices to construct an example of a strongly big-connecting family of effective $k$-cycles on $X$.  Fix a very ample divisor $A$ on $X$ and an $(n-k+1)$-dimensional subspace $V \subset |A|$.  Consider the family of effective $k$-cycles $p: U \to \mathbb{P}(V^{\vee})$ defined by taking intersections of $(n-k)$ general elements of $V$.  Choose an element $A \in |V|$.  A general cycle in our family $p$ only intersects $A$ along the base locus of $V$; thus any $p$-horizontal component of $A$ must be mapped under $s$ to the base-locus of $V$ which has dimension $k-1$.  So $p$ is a strongly big-connecting family for the divisor $A$.

Conversely, it suffices to show that a strongly big-connecting family $p$ has big class $\beta$.  We may restrict our family $p$ to (an open subset of) a general complete intersection of very ample divisors on $W$ to ensure that $\dim U = \dim X$ without changing the strongly big-connecting property.  We may then modify $p$ as in Lemma \ref{goodfamilymodification} to make $U$ and $W$ projective without changing the strongly big-connecting property. 

Then $s^{*}A$ is a big effective Cartier divisor on $U$.  Lemma \ref{stramphelp} shows that there is a big effective $k$-cycle $Z$ on $U$ whose support is contained in fibers of $p$ and in $p$-horizontal components of $s^{*}A$.  The former components are dominated by the cycle-theoretic fibers of $p$; the latter push forward to $0$.  Thus there is some constant $d$ such that $d\beta - {s}_{*}[Z]$ is a pseudo-effective class.   Since surjective pushforwards preserve bigness by \cite[Corollary 3.22]{fl13}, $s_{*}[Z]$ is a big class on $X$.  Thus we find that $\beta$ is also a big class.    
\end{proof}

\begin{exmple}
Suppose that $X$ is a smooth variety and $D$ is an irreducible divisor.  Theorem \ref{stronglyampleisbig} is similar to the fact that $D$ is big if $D|_{D}$ is ample (see \cite[Lemma 2.3]{voisin10}).
\end{exmple}

\begin{thrm} \label{variationandbigness}
Let $X$ be an integral projective variety and let $\alpha \in N_{k}(X)_{\mathbb{Q}}$.  Then the following conditions are equivalent:
\begin{enumerate}
\item $\ratvar(\alpha)>0$.
\item $\var(\alpha) > 0$.
\item There is an integral $(k+1)$-dimensional projective variety $Y$, a big class $\beta \in N_{k}(Y)_{\mathbb{Q}}$, and a morphism $f: Y \to X$ that is generically finite onto its image such that some multiple of $\alpha - f_{*}\beta$ is represented by an effective cycle.
\end{enumerate}
\end{thrm}

\begin{rmk}
The proof shows that Theorem \ref{variationandbigness} is also true if the morphism $f$ in (3) is instead required to be a closed immersion.
\end{rmk}

\begin{proof}
The case when $k=0$ is explained in Example \ref{0cyclevol}, so we may assume $k \geq 1$.

$(1) \implies (2)$ is obvious.

We next show $(2) \implies (3)$.  Suppose that $\var(\alpha) > 0$.  We may rescale $\alpha$ so that $\alpha \in N_{k}(X)_{\mathbb{Z}}$ and every positive multiple of $\alpha$ is represented by an effective cycle.  Fix a very ample Cartier divisor $A$, and choose some positive integer $m$ sufficiently large so that
\begin{equation*}
\chdim(m\alpha) >  \left( \begin{array}{c} m\alpha\cdot A^{k}+k \\ k \end{array} \right) + (m\alpha \cdot A^{k})k(n-k+1).
\end{equation*}
Let $p: U \to W$ denote a family of effective $k$-cycles that has maximal Chow dimension among all the families representing $m\alpha$.  Denote the projection map to $X$ by $s: U \to X$.  By replacing $A$ by a linearly equivalent divisor, we may suppose that $A$ that does not contain any component of $s(U)$.

Let $q: R \to W^{0}$ denote the intersection family of $p$ with $A$ as in Construction \ref{intersectionfamilyconstr} (where $W^{0}$ is an appropriately chosen open set of $W$).  The family $q$ has class $\beta := m\alpha \cdot A \in N_{k-1}(X)$.  By Theorem \ref{familydim}, we have
\begin{equation*}
\chdim(\beta)  < \chdim(p).
\end{equation*}
Thus there is a curve $T \subset W^{0}$ through a general point of $W$ that is contracted by $\ch_{q}$ but not by $\ch_{p}$.  Let $q_{T}: R_{T} \to T$ denote the restriction of the family to (an open subset of) $T$.  Using Lemma \ref{goodfamilymodification} we may extend the family $q_{T}$ to a projective closure of $T$; from now on we let $q_{T}$ denote this family over a projective base.

By Lemma \ref{injectivechowlemma} there is some irreducible component $R'$ of $R_{T}$ whose $s$-image in $X$ has dimension $k+1$.  Set $Y = s(R')$ with the reduced induced structure and $f: Y \to X$ the corresponding closed immersion.  By construction every $q_{T}$-horizontal component of $s^{*}A$ on $R'$ has $s$-image of dimension at most $k-1$.  Thus $q_{T}': R' \to T$ defines a strongly big-connecting family of divisors on $Y$ with respect to $A$.  If we set $\beta'$ to be the class on $Y$ of the family $q_{T}'$, then $\beta'$ is a big class on $Y$ by Theorem \ref{stronglyampleisbig}.  Furthermore $m\alpha - f_{*}\beta'$ is the class of an effective cycle.  Setting $\beta = \frac{1}{m}\beta'$ finishes the second implication.

To show $(3) \implies (1)$, suppose that there is a generically finite morphism $f: Y \to X$ from an integral $(k+1)$-dimensional projective variety $Y$ and a big class $\beta \in N_{k}(Y)_{\mathbb{Q}}$ so that some multiple of $\alpha - f_{*}\beta$ is represented by an effective cycle.  Let $V = f(Y)$ with the induced reduced structure, so we have morphisms $f': Y \to V$ and $i: V \subset X$.  Since surjective pushforwards preserve bigness by \cite[Corollary 3.22]{fl13}, $f'_{*}\beta$ is big on $V$ so that $\ratvar(f'_{*}\beta) > 0$.  By Lemma \ref{variationinclusionlem} $\ratvar(i_{*}f'_{*}\beta) > 0$.  Thus $\ratvar(\alpha) > 0$ as well.
\end{proof}

\nocite{*}
\bibliographystyle{amsalpha}
\bibliography{asymchow}

\end{document}